\documentclass[12pt, reqno, twoside, letterpaper]{amsart}

\usepackage{paperstyle}
\usepackage{graphicx}

\usepackage{mathtools}
\usepackage{todonotes}
\usepackage[norefs,nocites]{refcheck}

\newcommand{\be}{\begin{equation}}
\newcommand{\ee}{\end{equation}}
\newcommand{\dalign}[1]{\[\begin{aligned} #1 \end{aligned}\]}

\newcommand{\euB}{\EuScript{B}}

\newcommand{\euV}{\EuScript{V}}
\newcommand{\euZ}{\EuScript{Z}}
\newcommand{\er}{\mathrm{e}}

\title[Towards the GRH using zeros of the Riemann zeta function]
{Towards the Generalized Riemann Hypothesis using only zeros of the Riemann zeta function}
      
\author[W.~Banks]{William Banks}

\address{Department of Mathematics, 
         University of Missouri, 
         Columbia MO 65211, USA.}

\email{bankswd@missouri.edu}
        
\date{\today}

\begin{document}

\begin{abstract}
For any real $\beta_0\in[\tfrac12,1)$, let
${\rm GRH}[\beta_0]$ be the assertion that
for every Dirichlet character $\chi$ and all zeros
$\rho=\beta+i\gamma$ of $L(s,\chi)$, one has $\beta\le\beta_0$
(in particular, ${\rm GRH}[\frac12]$ is the Generalized Riemann Hypothesis).
In this paper, we show that the validity of ${\rm GRH}[\frac{9}{10}]$
depends only on certain distributional properties of the zeros of the
Riemann zeta function $\zeta(s)$. No conditions are imposed on the zeros
of nonprincipal Dirichlet $L$-functions.
\end{abstract}



\thanks{Data Availability Statement: Data sharing not applicable to this article as no datasets
were generated or analysed during the current study.}

\thanks{Potential Conflicts of Interest: NONE}

\thanks{Research Involving Human Participants and/or Animals: NONE}
\maketitle

\maketitle

%%%%%%%%%%%%%%%%%%%%%%%%
%%%%% PAPER BEGINS %%%%%
%%%%%%%%%%%%%%%%%%%%%%%%

\tableofcontents

%\centerline{\it Dedicated to John Friedlander and Henryk Iwaniec}

{\large\section{Introduction and statement of results}
\label{sec:intro}}

The \emph{Riemann zeta function} is a central
object of study in analytic number theory. In terms of 
the complex parameter $s=\sigma+it$, the zeta function is defined in the
half-plane $\sigma>1$ by two equivalent expressions:
$$
\zeta(s)\defeq\sum_{n\ge 1} n^{-s}
=\prod_{p\text{~prime}}(1-p^{-s})^{-1}.
$$
Riemann~\cite{Riemann} showed
that $\zeta(s)$ extends analytically to a meromorphic
function in the whole complex plane, its only singularity being a
simple pole at $s=1$. Moreover, the zeta function
satisfies a functional equation relating its values at
$s$ and $1-s$.\footnote{There are many excellent accounts of the
theory of the Riemann zeta function; we refer the reader to 
Titchmarsh~\cite{Titchmarsh} and to Borwein \emph{et al}~\cite{Borwein}
for essential background.}
The \emph{Riemann Hypothesis} (RH) asserts that if $\rho=\beta+i\gamma$
is a zero of $\zeta(s)$ with real part $\beta>0$, then $\beta=\frac12$.

More generally, for a Dirichlet character $\chi\bmod q$,
the Dirichlet \text{$L$-function} $L(s,\chi)$ is defined 
for $\sigma>1$ by:
$$
L(s,\chi)\defeq\sum_{n\ge 1} \chi(n)n^{-s}
=\prod_{p\text{~prime}}(1-\chi(p)p^{-s})^{-1}.
$$
The function $L(s,\chi)$ extends to a meromorphic
function (which is entire if $\chi$ is nonprincipal),
and when $\chi$ is primitive
it satisfies a simple functional equation relating its values at
$s$ and $1-s$; see, e.g., Bump~\cite[Chapter~1]{Bump}.
The \emph{Generalized Riemann Hypothesis} (GRH), which was perhaps first
formulated by Piltz in 1884 (see Davenport~\cite{Daven}), asserts that 
if $\rho=\beta+i\gamma$ is a zero of $L(s,\chi)$ with
$\beta>0$, then $\beta=\frac12$.

For any \emph{principal} character $\chi_0\bmod q$ one has
$$
L(s,\chi_0)=\zeta(s)\prod_{p\,\mid\,q}(1-p^{-s}),
$$
hence RH is equivalent to GRH for $L(s,\chi_0)$. On the other hand,
for \emph{nonprincipal} characters $\chi$, no direct relationship between
RH and GRH for $L(s,\chi)$ has been previously established. To establish
such a connection, we study the following weak form of the GRH
for Dirichlet $L$-functions.

\bigskip\noindent{\sc Hypothesis ${\rm GRH}[\beta_0]$}:
{\it Given $\beta_0\in[\frac12,1)$, the
inequality $\beta\le\beta_0$ holds for all zeros $\rho=\beta+i\gamma$ of
an arbitrary Dirichlet $L$-function $L(s,\chi)$.}

\bigskip\noindent
Note that ${\rm GRH}[\tfrac12]$ is equivalent to the
assertion that GRH holds for all Dirichlet \text{$L$-functions}.
In the present paper, we show that hypothesis ${\rm GRH}[\frac{9}{10}]$
can be reformulated entirely in terms of certain distributional
properties of the zeros of the Riemann zeta function.

To state our results, we introduce some notation. If $T>0$ is not the
ordinate of a zero of the zeta function, then $N(T)$ is used to denote the number
of zeros $\rho=\beta+i\gamma$ of $\zeta(s)$ in the rectangle $0<\beta<1$,
$0<\gamma<T$, and we define
$$
S(T)\defeq\frac1\pi{\rm arg}\,\zeta(1/2+iT).
$$
If $\gamma>0$ is the ordinate of a zero, then we set
$$
N(\gamma)\defeq\frac12\{N(\gamma^+)+N(\gamma^-)\},\qquad
S(\gamma)\defeq\frac12\{S(\gamma^+)+S(\gamma^-)\}.
$$
Using an explicit form of the well known relation
(see, e.g., Montgomery and Vaughan\cite[Corollary 14.2]{MontVau})
\be\label{eq:NTSTrelation}
N(T)=\frac{T}{2\pi}\log\frac{T}{2\pi}-\frac{T}{2\pi}+\frac{7}{8}
+S(T)+O(T^{-1})\qquad(T>0),
\ee
one sees that the difference $S(\gamma^+)-S(\gamma^-)$ is an integer
for every zero $\rho=\beta+i\gamma$ of $\zeta(s)$ with $\gamma>0$. Extending
the definition of $S(T)$ appropriately, this also holds for complex zeros
with $\gamma<0$. Therefore, writing
$$
\e(u)\defeq \er^{2\pi iu}\qquad(u\in\R),
$$
we can unambiguously define
\be\label{eq:Zrho-defn}
\euZ(\rho)\defeq\lim\limits_{T\to\gamma}\overline{\e\big(S(T)\big)}
\ee
(the common value of the left and right limits)
for every complex zero $\rho$ of $\zeta(s)$.

Throughout the paper, $C_c^\infty(\R^+)$ is the space of smooth
functions $f:\R^+\to\C$ with compact support in $\R^+$.
Let $\mu$ and $\phi$ denote the M\"obius and Euler functions, respectively.

\bigskip

\begin{theorem}\label{thm:RHvsGRH}
Assume RH. Suppose that for every function $\euB\in C_c^\infty(\R^+)$
and every rational number $\xi\defeq m/q$ with $0<m<q$ and $(m,q)=1$, the bound
\be\label{eq:superbound}
\sum_{\rho=\frac12+i\gamma}\xi^{-i\gamma}\,\euZ(\rho)
\euB\Big(\frac{\gamma}{2\pi X}\Big)
+\frac{\mu(q)}{\phi(q)}\sum_{n\ge 1}\Lambda(n)\euB(n/X)
\,\mathop{\ll}\limits_{\xi,\euB,\eps}\,X^{9/10+\eps}
\ee
holds for any $\eps>0$, where the sum on the left side runs over all
complex zeros $\rho=\frac12+i\gamma$ of $\zeta(s)$,
the quantity $\euZ(\rho)$ is given by \eqref{eq:Zrho-defn},
and the implied constant depends only on
$\xi$, $\euB$, and $\eps$. Then, the hypothesis ${\rm GRH}[\frac{9}{10}]$ is true.
\end{theorem}

We reiterate that the hypotheses of Theorem~\ref{thm:RHvsGRH}
involve only properties of the zeros of the Riemann zeta function.
\emph{No conditions are imposed on the zeros
of nonprincipal Dirichlet $L$-functions}.

We also have the following converse of Theorem~\ref{thm:RHvsGRH}.

\begin{theorem}\label{thm:RHvsGRH2}
Assume RH. If ${\rm GRH}[\frac{9}{10}]$ is true, then
for every $\euB\in C_c^\infty(\R^+)$ and 
every rational number $\xi\defeq m/q$ with $0<m<q$ and $(m,q)=1$, the bound
\eqref{eq:superbound} holds for any $\eps>0$.
\end{theorem}

Throughout the paper, implied constants in the
symbols $\ll$, $O$, etc., often depend on various parameters
(e.g., $\xi$, $\euB$, $\eps$) as indicated by the notation (e.g.,
see \eqref{eq:superbound} above); these constants
are independent of all other parameters.

{\large\section{Approach}}

Using the explicit formula and assuming RH, we show that a sum of the form
\be\label{eq:einstein}
\sum_{n\ge 1}\Lambda(n)\e(-n\xi)\euB(n/X)
\ee
with $\xi\in\R^+$ and $\euB\in C_c^\infty(\R^+)$ is
equal to
\be\label{eq:feynman}
-\ssum{\rho=\frac12+i\gamma}\xi^{-1/2-i\gamma}
\euZ(\rho)\euB\Big(\frac{\gamma}{2\pi\xi X}\Big)
\ee
up to an error of at most $O_{\xi,\euB}(X^{9/10})$;
see Theorem~\ref{thm:vonMangoldt-twist}.
To handle the integrals that
arise in our use of the explicit formula, we apply the method of the
stationary phase (however, some care is needed to
obtain adequate and explicit estimates for the error terms);
see Lemma~\ref{lem:the-child}, which is the main workhorse for
the proof of Theorem~\ref{thm:vonMangoldt-twist}.

Given the close relationship (under RH) between the sums \eqref{eq:einstein}
and \eqref{eq:feynman}, the estimate \eqref{eq:superbound} for
rational numbers $\xi\defeq m/q$ allows us to deduce an equally strong bound
\be\label{eq:thorne}
\sum_{n\ge 1}\Lambda(n)\chi(n)\euB(n/X)
\,\mathop{\ll}\limits_{q,\euB,\eps}\, X^{9/10+\eps}.
\ee
for any primitive Dirichlet character $\chi$ modulo $q>1$;
see the proof of Theorem~\ref{thm:RHvsGRH} in \S\ref{sec:pfThm1}.
Since $\euB\in C_c^\infty(\R^+)$ is arbitrary,
${\rm GRH}[\frac{9}{10}]$ holds for the function $L(s,\chi)$.
Conversely, under ${\rm GRH}[\frac{9}{10}]$, it follows that
\eqref{eq:thorne} holds for every primitive character $\chi\bmod q$.
This leads to similar estimates for
the sums \eqref{eq:einstein} and \eqref{eq:feynman}; 
see the proof of Theorem~\ref{thm:RHvsGRH2} in \S\ref{sec:pfThm1}.

\bigskip

{\large\section{Integral bounds}
\label{sec:integralbounds}}

Our aim in this section is to establish certain integral bounds that are needed
in the proof of Theorem~\ref{thm:vonMangoldt-twist} (see
\S\ref{sec:vonMangoldt-twist}).

For each $\euB\in C_c^\infty(\R^+)$,
let $X_\euB\ge 10$ be a number large enough so that
$$
1\not\in X\,\text{\rm supp}(\euB)\defeq\big\{Xu:u\in\text{\rm supp}(\euB)\big\}
\qquad(X>X_\euB).
$$
The value of $X_\euB$ is held fixed throughout the paper.

\bigskip

\begin{lemma}\label{lem:mandalorian}
Let $\xi\in\R^+$, $\euB\in C_c^\infty(\R^+)$, and $X>X_\euB$. Then
$$
\int_{\R^+}\Big(1-\frac{1}{u^3-u}\Big)\e(-u\xi)\euB(u/X)\,du
\,\mathop{\ll}_{\xi,\euB}\,X^{-1}.
$$
\end{lemma}

\begin{proof}
Put
$$
g(u)\defeq\Big(1-\frac{1}{u^3-u}\Big)\euB(u/X)
\qquad(u\in\R^+).
$$
Then $g$ and all of its derivatives are smooth
and supported on the set $X\,\text{\rm supp}(\euB)$
(note that $\euB(u/X)=0$ for $u\in\{0,\pm 1\}$ since
$\euB\in C_c^\infty(\R^+)$ and $X>X_\euB$).
In particular, $g''(u)=0$ unless $u\,\mathop{\asymp}_\euB\, X$, in
which case one has $g''(u)\,\mathop{\ll}_\euB\,X^{-2}$. 
Integrating by parts twice, we have (since $g$ and $g'$ vanish at $0$ and $\infty$)
$$
\int_{\R^+}\e(-u\xi)g(u)\,du
=\int_{\R^+}
\frac{\e(-u\xi)}{(-2\pi i\xi)^2}\,g''(u)\,du
\,\mathop{\ll}_\euB\,\xi^{-2}X^{-2}
\,{\rm meas}\big(X\,\text{\rm supp}(\euB)\big),
$$
and the lemma follows.
\end{proof}

The next two technical lemmas are needed in the proof of
Lemma~\ref{lem:the-child} below.

\begin{lemma}\label{lem:technical}
Fix $u_0\in\R^+$, and let $\cL,\cD:C_c^\infty(\R^+)\to C_c^\infty(\R^+)$
be the linear operators defined by
$$
\cL F(u)\defeq\frac{uF(u)}{(u_0-u)}
\mand
\cD F(u)\defeq F'(u)\qquad(u\in\R^+).
$$
For any integer $k\ge 0$, let $(\cD\cL)^k$ and $\cL(\cD\cL)^k$ be the linear
operators given by
$$
[\cD\cL]^k\defeq
\mathop{\underbrace{~\cD\circ\cL\circ\cdots\circ\cD\circ\cL~}}
\limits_{\text{\rm $k$ copies each of $\cD$ and $\cL$}}
\mand
\cL[\cD\cL]^k\defeq\cL\circ[\cD\cL]^k.
$$
Then for every $F\in C_c^\infty(\R^+)$ and every integer $k\ge 0$ we have
\dalign{
[\cD\cL]^kF(u)&\,\mathop{\ll}_{F,k}\,
\max\big\{|u-u_0|^{-k},|u-u_0|^{-2k}\big\},\\
\cL[\cD\cL]^kF(u)&\,\mathop{\ll}_{F,k}\,
\max\big\{|u-u_0|^{-k},|u-u_0|^{-2k-1}\big\},
}
for all $u\in\R^+$, $u\ne u_0$.
\end{lemma}

\begin{proof}
Let $F\in C_c^\infty(\R^+)$ be fixed in what follows.
For any integers $A,B\ge 0$ and a real number $C>0$,
let $\euV(A,B,C)$ denote the set of functions $G\in C_c^\infty(\R^+)$
of the form
$$
G(u)=\ssum{h,i,j\ge 0\\h\le A\le j\\i+j=h+B}
c_{h,i,j}\frac{u^hF^{(i)}(u)}{(u_0-u)^j},
$$
where the sum runs over nonnegative integers $h,i,j$,
and the coefficients $c_{h,i,j}$ are complex numbers
satisfying $|c_{h,i,j}|\le C$. If $F$ is supported on the
interval $[a,b]$, where $0<a<b<\infty$, then the trivial bound
\be\label{eq:eureka0}
\big|G(u)\big|\le C'_{A,B}\,C\max\{b^A,1\}
\max\limits_{0\le i\le A+B}\big|F^{(i)}(u)\big|
\max_{A\le j\le A+B}|u-u_0|^{-j}
\ee
holds for all $G\in\euV(A,B,C)$, where
$$
C'_{A,B}\defeq\big|\{(h,i,j):h,i,j\ge 0,~h\le A\le j,~i+j=h+B\}\big|
\le\sum_{h=0}^A\sum_{j=A}^{A+B}1\,\mathop{\ll}\limits_{A,B}\,1.
$$

Next, noting that
$$
\cL G(u)=\ssum{h,i,j\ge 0\\h\le A\le j\\i+j=h+B}c_{h,i,j}
\frac{u^{h+1}F^{(i)}(u)}{(u_0-u)^{j+1}}
=\ssum{h\ge 1,\,i,j\ge 0\\h\le A+1\le j\\i+j=h+B}c_{h-1,i,j-1}
\frac{u^hF^{(i)}(u)}{(u_0-u)^j},
$$
it follows that
\be\label{eq:eureka1}
\cL:\euV(A,B,C)\to\euV(A+1,B,C).
\ee
Similarly, we write $\cD G(u)=G_1(u)+G_2(u)+G_3(u)$, 
where the functions $G_j$ are defined as follows. First, we have
$$
G_1(u)\defeq\ssum{h\ge 1,\,i,j\ge 0\\h\le A\le j\\i+j=h+B}c_{h,i,j}
\frac{hu^{h-1}F^{(i)}(u)}{(u_0-u)^j}
=\ssum{h,i,j\ge 0\\h+1\le A\le j\\i+j=h+B+1}(h+1)c_{h+1,i,j}
\frac{u^hF^{(i)}(u)}{(u_0-u)^j}.
$$
Since the coefficients $(h+1)c_{h+1,i,j}$ do not exceed $AC$ in absolute
value, it is clear that $G_1\in\euV(A,B+1,AC)$. Next, we have
$$
G_2(u)\defeq\ssum{h,i,j\ge 0\\h\le A\le j\\i+j=h+B}c_{h,i,j}
\frac{u^hF^{(i+1)}(u)}{(u_0-u)^j}
=\ssum{h,j\ge 0,\,i\ge 1\\h\le A\le j\\i+j=h+B+1}c_{h,i+1,j}
\frac{u^hF^{(i)}(u)}{(u_0-u)^j},
$$
and clearly $G_2\in\euV(A,B+1,C)$. Finally, we have
$$
G_3(u)\defeq\ssum{h,i,j\ge 0\\h\le A\le j\\i+j=h+B}c_{h,i,j}
\frac{ju^hF^{(i)}(u)}{(u_0-u)^{j+1}}
=\ssum{h,i\ge 0,\,j\ge 1\\h\le A\le j-1\\i+j=h+B+1}(j-1)c_{h,i,j-1}
\frac{u^hF^{(i)}(u)}{(u_0-u)^j},
$$
As the coefficients $(j-1)c_{h,i,j-1}$ do not exceed $(A+B)C$ in absolute
value, we have $G_3\in\euV(A,B+1,(A+B)C)$. Consequently,
$\cD G\in\euV(A,B+1,(2A+B+1)C)$, and therefore
\be\label{eq:eureka2}
\cD:\euV(A,B,C)\to\euV\big(A,B+1,(2A+B+1)C\big).
\ee

Observe that $F$ itself lies in $\euV(0,0,1)$. Therefore,
an inductive argument using \eqref{eq:eureka1} and \eqref{eq:eureka2} shows that
for every integer $k\ge 0$ we have
$$
[\cD\cL]^kF\in\euV(k,k,k!\cdot 3^k)
\mand
\cL[\cD\cL]^kF\in\euV(k+1,k,k!\cdot 3^k).
$$
Using \eqref{eq:eureka0}, the result follows.
\end{proof}

\begin{lemma}\label{lem:funnyexpint}
For any positive number $\lambda$, we have
$$
\mathop{\int_\R}_{(|u|>\lambda)}\er^{-iu^2}\,du
=\frac{\er^{-i\lambda^2}}{i\lambda}+i
\mathop{\int_\R}_{(|u|>\lambda)}\er^{-iu^2} u^{-2}\,du
\ll\lambda^{-1}.
$$
\end{lemma}

The next result is our primary tool; its proof is based on the well known
stationary phase method.

\begin{lemma}\label{lem:the-child}
Let $\xi\in\R^+$, $\euB\in C_c^\infty(\R^+)$, and $X>X_\euB$. Suppose
 $\euB$ is supported on the interval $[a,b]$, where $0<a<b<\infty$.
For any real number $\gamma\ne 0$, consider the integral $\cI(\gamma)$
defined by
$$
\cI(\gamma)\defeq\int_{\R^+}\e(-u\xi)\euB(u/X)u^{-1/2+i\gamma}\,du.
$$
Put
\be\label{eq:ast-bst-defn}
a_\star\defeq\frac{1+2a-\sqrt{1+4a}}{2},\qquad
b_\star\defeq\frac{1+2b+\sqrt{1+4b}}{2},\qquad
\gamma_\star\defeq\frac{\gamma}{2\pi\xi X}.
\ee
Then, for $\gamma_\star\not\in[a_\star,b_\star]$ the bound
$$
\cI(\gamma)\,\mathop{\ll}_{\xi,\euB}\,X^{1/2}
\max\big\{X^{-2}|\gamma|^{-2},|\gamma|^{-4}\big\}
$$
holds, whereas for $\gamma_\star\in[a_\star,b_\star]$ we have the estimate
$$
\cI(\gamma)=\xi^{-1/2-i\gamma}
\e\Big(\frac{\gamma}{2\pi}\log\frac{\gamma}{2\pi\er}+\frac78\Big)
\euB(\gamma_\star)+O_{\xi,\euB}(X^{-1/10}).
$$
\end{lemma}

\begin{proof}
Making the change of variables $u\mapsto Xu$, we have
\be\label{eq:solar1}
\cI(\gamma)=X^{1/2+i\gamma}\int_{\R^+}\er^{iXf(u)}g(u)\,du
=X^{1/2+i\gamma}\,\cJ\quad\text{(say)},
\ee
where
$$
f(u)\defeq -2\pi\xi u+\frac{\gamma\log u}{X},\qquad
g(u)\defeq \frac{\euB(u)}{u^{1/2}}.
$$
Note that $\gamma_\star=\gamma/(2\pi\xi X)$
is the only real number for which $f'(\gamma_\star)=0$.

For any given $\Delta>0$, we write
$\cJ=\cJ_\infty+\cJ_\star$
with
$$
\cJ_\infty\defeq\mathop{\int_{\R^+}}_{(|u-\gamma_\star|>\Delta)}
\er^{iXf(u)}g(u)\,du,\qquad
\cJ_\star\defeq\mathop{\int_{\R^+}}_{(|u-\gamma_\star|\le\Delta)}
\er^{iXf(u)}g(u)\,du.
$$
We study $\cJ_\infty$ first.
Let $\cL$ and $\cD$ be the operators defined in
Lemma~\ref{lem:technical} with $u_0\defeq\gamma_\star$,
and put
$$
g_k\defeq [\cD\cL]^kg
\mand
\tilde g_k\defeq \cL[\cD\cL]^kg\qquad(k\ge 0).
$$
According to Lemma~\ref{lem:technical},
\be\label{eq:signals}
g_k(u)\,\mathop{\ll}_{\euB,k}\,
\max\big\{|u-\gamma_\star|^{-k},|u-\gamma_\star|^{-2k}\big\}
\qquad(u\ne\gamma_\star).
\ee
Taking into account that $f'(u)=2\pi\xi u^{-1}(\gamma_\star-u)$, we have
$$
\cJ_\infty=\frac{1}{2\pi\xi}\mathop{\int_{\R^+}}_{(|u-\gamma_\star|>\Delta)}
\er^{iXf(u)}f'(u)\cdot\tilde g_0(u)\,du
=-\frac{1}{2\pi i\xi X}\mathop{\int_{\R^+}}_{(|u-\gamma_\star|>\Delta)}
\er^{iXf(u)} g_1(u)\,du,
$$
where we have used integration by parts in the last step.
Similarly, by induction on $k$, we see that
\be\label{eq:signals2}
\cJ_\infty=\frac{1}{(-2\pi i\xi X)^k}\mathop{\int_{\R^+}}_{(|u-\gamma_\star|>\Delta)}
\er^{iXf(u)} g_k(u)\,du.
\ee

The numbers $a_\star$ and $b_\star$ defined in \eqref{eq:ast-bst-defn} have
the property that if $\gamma_\star\not\in[a_\star,b_\star]$, then
$|u-\gamma_\star|>|\gamma_\star|^{1/2}$ for all $u\in[a,b]$. Hence, choosing
$\Delta\defeq|\gamma_\star|^{1/2}$ in the case that
$\gamma_\star\not\in[a_\star,b_\star]$, it follows that $\cJ_\infty=\cJ$
and $\cJ_\star=0$. Setting $k\defeq 4$ and combining \eqref{eq:signals}
and \eqref{eq:signals2}, we derive the bound
$$
\cJ\,\mathop{\ll}_\euB\,(\xi X)^{-4}
\max\big\{|\gamma_\star|^{-2},|\gamma_\star|^{-4}\big\}
\ll\max\big\{(\xi X)^{-2}|\gamma|^{-2},|\gamma|^{-4}\big\}
$$
In view of \eqref{eq:solar1}, we obtain the first statement of the lemma.

From now on, we assume that $\gamma_\star\in[a_\star,b_\star]$.
We further assume that
$X$ is large enough (depending on $\euB$) so that
\be\label{eq:Delta-constraint}
\Delta\defeq X^{-2/5}<\min\{1,2^{-1/2}a_\star\}.
\ee
Write $\cJ=\cJ_\infty+\cJ_\star$ as before with this $\Delta$.
Setting $k\defeq 5$ (say) and combining \eqref{eq:signals}
and \eqref{eq:signals2}, we derive the bound
\be\label{eq:solar2}
\cJ_\infty\,\mathop{\ll}_{\xi,\euB}\,X^{-1}.
\ee

Turning to the estimate of $\cJ_\star$, observe that the upper bound
\eqref{eq:Delta-constraint} implies that the interval
$\Omega_\Delta\defeq[\gamma_\star-\Delta,\gamma_\star+\Delta]$ lies
entirely inside $\R^+$; in particular, 
$$
\cJ_\star\defeq\int_{\Omega_\Delta}\er^{iXf(u)}g(u)\,du.
$$
Since $\gamma_\star\in[a_\star,b_\star]$ (and thus,
$\gamma_\star\,\mathop{\asymp}_\euB\,1)$, the estimate
$$
f'''(u)=\frac{4\pi\xi\gamma_\star}{u^3}\,\mathop{\asymp}_\euB\,\xi
$$
holds uniformly
for all $u\in\Omega_\Delta$, hence by Taylor's approximation we have
$$
\er^{iXf(u)}=\er^{iX\{f(\gamma_\star)
+\frac12f''(\gamma_\star)(u-\gamma_\star)^2\}}(1+r_1(u))
$$
for some complex function $r_1$ such that
$$
r_1(u)\,\mathop{\ll}_{\xi,\euB}\,X|u-\gamma_\star|^3
\qquad(u\in\Omega_\Delta).
$$
We can also write
$$
g(u)=g(\gamma_\star)+r_2(u),
$$
where $r_2$ satisfies the bound
$$
r_2(u)\,\mathop{\ll}_{\xi,\euB}\,|u-\gamma_\star|\qquad(u\in\Omega_\Delta).
$$
Therefore, since
\dalign{
\int_{\Omega_\Delta}\big|r_1(u)\big|\,du
&\,\mathop{\ll}_{\xi,\euB}\,X\int_{\gamma_\star-\Delta}^{\gamma_\star+\Delta}
|u-\gamma_\star|^3\,du\ll X\Delta^4,\\
\int_{\Omega_\Delta}\big|r_2(u)\big|\,du
&\,\mathop{\ll}_{\xi,\euB}\,\int_{\gamma_\star-\Delta}^{\gamma_\star+\Delta}
|u-\gamma_\star|\,du\ll\Delta^2,
}
and $\Delta^2\ll X\Delta^4=X^{-3/5}$ by \eqref{eq:Delta-constraint},
we derive the estimate
\begin{equation}
\label{eq:solar3}
\begin{split}
\cJ_\star
&=\int_{\Omega_\Delta}\er^{iX\{f(\gamma_\star)
+\frac12f''(\gamma_\star)(u-\gamma_\star)^2\}}(1+r_1(u))
\,(g(\gamma_\star)+r_2(u))\,du\\
&=\er^{-i\gamma}\gamma_\star^{-1/2+i\gamma}\euB(\gamma_\star)
\int_{\Omega_\Delta}
\er^{-\pi i\xi X\gamma_\star^{-1}(u-\gamma_\star)^2}\,du
+O_{\xi,\euB}(X^{-3/5}),
\end{split}
\end{equation}
where in the last step we used the identities
$$
Xf(\gamma_\star)=-\gamma+\gamma\log\gamma_\star,\qquad
f''(\gamma_\star)=-2\pi\xi\gamma_\star^{-1},\qquad
g(\gamma_\star)\defeq \gamma_\star^{-1/2}\euB(\gamma_\star).
$$

Next, we extend the range of integration
in the preceding integral to all of $\R$ (with
an acceptable error). Consider the integral
$$
\cK\defeq\mathop{\int_\R}_{(u\not\in\Omega_\Delta)}
\er^{-\pi i\xi X\gamma_\star^{-1}(u-\gamma_\star)^2}\,du.
$$
Making the change of variables $u\mapsto c\,u+\gamma_\star$,
where $c\defeq\sqrt{\gamma_\star/(\pi\xi X)}$, and applying
Lemma~\ref{lem:funnyexpint}, we have
$$
\cK=c\mathop{\int_\R}_{(|u|>\Delta/c)}
\er^{-iu^2}\,du\ll\,c^2\Delta^{-1}\,\mathop{\ll}_{\xi,\euB}\,X^{-3/5}.
$$
Using this bound together with \eqref{eq:solar3}, it follows that
$$
\cJ_\star=\er^{-i\gamma}\gamma_\star^{-1/2+i\gamma}\euB(\gamma_\star)
\int_\R\er^{-\pi i\xi X\gamma_\star^{-1}(u-\gamma_\star)^2}\,du
+O_{\xi,\euB}(X^{-3/5}).
$$

Combining the previous bound with \eqref{eq:solar2}, we have
\be\label{eq:solar5}
\cJ=\er^{-i\gamma}\gamma_\star^{-1/2+i\gamma}\euB(\gamma_\star)
\int_\R\er^{-\pi i\xi X\gamma_\star^{-1}(u-\gamma_\star)^2}\,du
+O_{\xi,\euB}(X^{-3/5}).
\ee
The integral here can be explicitly evaluated: 
$$
\int_\R\er^{-\pi i\xi X\gamma_\star^{-1}(u-\gamma_\star)^2}\,du
=\er^{-i\pi/4}\sqrt{\frac{\gamma_\star}{\xi X}}.
$$
Inserting this result into \eqref{eq:solar5}
and recalling \eqref{eq:solar1}, after some
simplification we find that
$$
\cI(\gamma)=\xi^{-1/2-i\gamma}
\e\Big(\frac{\gamma}{2\pi}\log\frac{\gamma}{2\pi\er}-\frac18\Big)
\euB(\gamma_\star)
+O_{\xi,\euB}(X^{-1/10}),
$$
and the proof is complete.
\end{proof}

{\large\section{Twisting the von Mangoldt function}
\label{sec:vonMangoldt-twist}}

\begin{theorem}\label{thm:vonMangoldt-twist}
Assume RH.
Let $\xi\in\R^+$, $\euB\in C_c^\infty(\R^+)$, and $X>X_\euB$. Then
$$
\sum_{n\ge 1}\Lambda(n)\e(-n\xi)\euB\Big(\frac{n}{X}\Big)
=-\ssum{\rho=\frac12+i\gamma}\xi^{-1/2-i\gamma}
\euZ(\rho)\euB\Big(\frac{\gamma}{2\pi\xi X}\Big)
+O_{\xi,\euB}(X^{9/10}).
$$
\end{theorem}

\begin{proof}
Our goal is to estimate
$$
\sum_{n\ge 1}\Lambda(n)\e(-n\xi)\euB\Big(\frac{n}{X}\Big)
=\sum_{n\ge 1}\Lambda(n)\varphi(n),
$$
where $\varphi(u)\defeq\e(-u\xi)\euB(u/X)$.
By the explicit formula
(see, e.g., Iwaniec and Kowalski~\cite[Exercise~5, p.~109]{IwaniecKowalski})
we have
\be\label{eq:explform}
\sum_{n\ge 1}\Lambda(n)\varphi(n)
=\int_{\R^+}\Big(1-\frac{1}{(u-1)u(u+1)}\Big)\varphi(u)\,du
-\sum_\rho\hat\varphi(\rho),
\ee
where $\hat\varphi$ is the Mellin transform of $\varphi$ given by
$$
\hat\varphi(s)\defeq\int_{\R^+} \varphi(u)u^{s-1}\,du.
$$
Using Lemma~\ref{lem:mandalorian} to bound the integral in \eqref{eq:explform},
we get that
\be\label{eq:FcX}
\sum_{n\ge 1}\Lambda(n)\varphi(n)
=-\sum_\rho\hat\varphi(\rho)+O_{\xi,\euB}(X^{-1}).
\ee

Next, for any complex zero $\rho=\tfrac12+i\gamma$ of $\zeta(s)$ we have
$$
\hat\varphi(\rho)\defeq\int_{\R^+} 
\e(-u\xi)\euB(u/X)u^{-1/2+i\gamma}\,du,
$$
which is the integral $\cI(\gamma)$ considered in
Lemma~\ref{lem:the-child}. Defining $a,a_\star,a,b_\star,\gamma_\star$
as in Lemma~\ref{lem:the-child}, it follows that
\be\label{eq:sydney}
\sum_{\rho=\frac12+i\gamma}\hat\varphi(\rho)
=\Sigma_1+O_{\xi,\euB}(X^{9/10}),
\ee
where
$$
\Sigma_1\defeq\ssum{\rho=\frac12+i\gamma\\
\gamma_\star\in[a_\star,b_\star]}\xi^{-1/2-i\gamma}
\e\Big(\frac{\gamma}{2\pi}\log\frac{\gamma}{2\pi\er}+\frac78\Big)
\euB(\gamma_\star).
$$
Note that the error term in \eqref{eq:sydney} is a consequence of
the following bounds on the sums that arise
naturally in our application of Lemma~\ref{lem:the-child}:
$$
\ssum{\rho=\frac12+i\gamma\\
\gamma_\star\not\in[a_\star,b_\star]}X^{1/2}
\max\big\{X^{-2}|\gamma|^{-2},|\gamma|^{-4}\big\}
\,\mathop{\ll}_{\xi,\euB}\,X^{1/2},
\qquad
\ssum{\rho=\frac12+i\gamma\\
\gamma_\star\in[a_\star,b_\star]}X^{-1/10}
\,\mathop{\ll}_{\xi,\euB}\,X^{9/10}.
$$
The condition ``$\gamma_\star\in[a_\star,b_\star]$'' in the above
definition of $\Sigma_1$ is redundant (indeed,
we have $a_\star<a<b<b_\star$ by \eqref{eq:ast-bst-defn}, hence
$\euB(\gamma_\star)=0$ if $\gamma_\star\not\in[a_\star,b_\star]$),
and so we can simply write
$$
\Sigma_1=\ssum{\rho=\frac12+i\gamma}\xi^{-1/2-i\gamma}
\e\Big(\frac{\gamma}{2\pi}\log\frac{\gamma}{2\pi\er}+\frac78\Big)
\euB\Big(\frac{\gamma}{2\pi\xi X}\Big).
$$
Next, observe that \eqref{eq:NTSTrelation} implies
$$
\e\Big(\frac{T}{2\pi}\log\frac{T}{2\pi\er}+\frac78\Big)
=\e\big(N(T)-S(T)+O(T^{-1})\big)
=\overline{\e\big(S(T)\big)}+O(T^{-1})
$$
provided that $T>0$ is not the ordinate of a zero of $\zeta(s)$
(since $N(T)\in\Z$). Taking the limit as $T\to\gamma$, we get that
$$
\e\Big(\frac{\gamma}{2\pi}\log\frac{\gamma}{2\pi\er}+\frac78\Big)
=\euZ(\rho)+O(\gamma^{-1}),
$$
where (as in \S\ref{sec:intro})
$$
\euZ(\rho)\defeq\lim\limits_{T\to\gamma}\overline{\e\big(S(T)\big)}.
$$
Thus, up to an acceptable error, we can replace $\Sigma_1$ in
\eqref{eq:sydney} with the quantity
$$
\Sigma_1\defeq\ssum{\rho=\frac12+i\gamma}\xi^{-1/2-i\gamma}
\euZ(\rho)\euB\Big(\frac{\gamma}{2\pi\xi X}\Big).
$$
The theorem now follows by combining \eqref{eq:FcX}
and \eqref{eq:sydney}.
\end{proof}

{\large\section{Proofs of Theorems~\ref{thm:RHvsGRH} and~\ref{thm:RHvsGRH2}}
\label{sec:pfThm1}}

We begin by considering the GRH for individual Dirichlet $L$-functions.
For convenience, we formulate the following hypothesis.

\bigskip\noindent{\sc Hypothesis ${\rm GRH}[\beta_0,\chi]$}:
{\it Given $\beta_0\in[\frac12,1)$ and a Dirichlet character $\chi$, the
inequality $\beta\le\beta_0$ holds for all zeros $\rho=\beta+i\gamma$ of
the $L$-function $L(s,\chi)$.}

\bigskip\noindent Note that ${\rm GRH}[\beta_0]$ is true if and only if
${\rm GRH}[\beta_0,\chi]$ holds for all characters~$\chi$.

\bigskip

\begin{lemma}\label{lem:ultraclean}
Fix $\beta_0\in[\tfrac12,1)$, and let
$\chi$ be a nonprincipal character of modulus $q$. Then the following are equivalent:
\begin{itemize}
\item[$(i)$] Hypothesis ${\rm GRH}[\beta_0,\chi]$ is true;
\item[$(ii)$] For any function $\euB\in C_c^\infty(\R^+)$, we have
$$
\sum_{n\ge 1}\Lambda(n)\chi(n)\euB(n/X)
\,\mathop{\ll}\limits_{q,\euB,\eps}\, X^{\beta_0+\eps}.
$$
\end{itemize}
\end{lemma}

\begin{proof}
Consider the intermediate bound
\be\label{eq:birdcall}
\sum_{n\le X}\Lambda(n)\chi(n)\,\mathop{\ll}\limits_{q,\eps}\,X^{\beta_0+\eps}.
\ee
The equivalence $(i)\Longleftrightarrow\eqref{eq:birdcall}$ is standard and well known,
and the implication $\eqref{eq:birdcall}\Longrightarrow(ii)$ is immediate using
partial summation. To prove $(ii)\Longrightarrow\eqref{eq:birdcall}$
one uses a \underline{fixed} test
function $\euB\in C_c^\infty(\R^+)$ such that $0\le\euB(u)\le 1$ for all $u\in\R$,
and $\euB(u)=1$ for all $u\in[\tfrac12,1]$. By $(ii)$ we have
$$
\sum_{X/2<n\le X}\hskip-5pt\Lambda(n)\chi(n)
=\sum_{X/2<n\le X}\hskip-5pt\Lambda(n)\chi(n)\euB(n/X)
\le\sum_{n\ge 1}\Lambda(n)\chi(n)\euB(n/X)
\,\mathop{\ll}\limits_{q,\eps}\, X^{\beta_0+\eps},
$$
and then \eqref{eq:birdcall} follows by a standard dyadic argument.
\end{proof}

\begin{lemma}\label{lem:aloevera}
Fix $\beta_0\in[\tfrac12,1)$. Let $\xi\defeq m/q$ with $0<m<q$ and $(m,q)=1$.
If ${\rm GRH}[\beta_0,\chi]$ holds for all nonprincipal characters
$\chi\bmod q$, then for any $\eps>0$ we have
\be\label{eq:meowing}
\sum_{n\ge 1}\Lambda(n)\e(-nm/q)\euB(n/X)
=\frac{\mu(q)}{\phi(q)}\sum_{n\ge 1}\Lambda(n)\euB(n/X)
+O_{\xi,\euB,\eps}(X^{\beta_0+\eps})
\ee
\end{lemma}

\begin{proof}
Introducing an error of at most $O_\euB(X^{1/2})$, the sum on the left
side of \eqref{eq:meowing} can be replaced by
\dalign{
&\sum_{(n,q)=1}\Lambda(n)\,\e(-nm/q)\euB(n/X)=\sum_{(a,q)=1}\e(-am/q)
\sum_{n\equiv a\bmod q}\Lambda(n)\euB(n/X)\\
&\qquad\qquad=\frac{1}{\phi(q)}\sum_{(a,q)=1}\e(-am/q)
\sum_{\chi\bmod q}\overline\chi(a)
\sum_{n\ge 1}\Lambda(n)\chi(n)\euB(n/X).
}
The contribution from the principal character $\chi_0$ is
$$
\frac{1}{\phi(q)}\sum_{(a,q)=1}\e(-am/q)
\sum_{n\ge 1}\Lambda(n)\euB(n/X)
=\frac{\mu(q)}{\phi(q)}\sum_{n\ge 1}\Lambda(n)\euB(n/X)
$$
(the first sum is a Ramanujan sum), and by Lemma~\ref{lem:ultraclean}\,$(ii)$
the total contribution from all nonprincipal characters
is at most $O_{q,\euB,\eps}(X^{\beta_0+\eps})$.
\end{proof}

We are now ready to prove Theorems~\ref{thm:RHvsGRH} and~\ref{thm:RHvsGRH2}.

\begin{proof}[Proof of Theorem~\ref{thm:RHvsGRH}]
Assume RH. Since RH is equivalent to GRH in the case of principal characters,
we can assume $\chi$ is nonprincipal. Moreover, if
$\chi$ is induced from a primitive character~$\chi_\star$
of conductor $q_*$, then $L(s,\chi)$ and $L(s,\chi_\star)$ have the same zeros
in the critical strip since
$$
L(s,\chi)=L(s,\chi_*)
\sprod{p\,\mid\,q\\p\,\nmid\,q_*}(1-\chi_*(p)p^{-s}).
$$
Hence, to prove Theorem~\ref{thm:RHvsGRH} it suffices to show
that ${\rm GRH}[\frac{9}{10},\chi]$ holds for any Dirichlet $L$-function
attached to a \emph{primitive} Dirichlet character $\chi$ of modulus $q>1$.
In this situation, the following identity is well known:
\be\label{eq:chi-basic}
\chi(n)=\frac{\chi(-1)\tau(\chi)}{q}
\sum_{m\bmod q}\overline\chi(m)\e(-nm/q)
\qquad(n\in\Z),
\ee
where the sum runs over any complete set of residue
classes $m\bmod q$, and
$\tau(\chi)$ is the Gauss sum given by
$$
\tau(\chi)\defeq\sum_{n\bmod q}\chi(n)\e(n/q);
$$
see, e.g., Bump~\cite[Chapter 1]{Bump}. In particular,
for every $\euB\in C_c^\infty(\R^+)$ we have
\be\label{eq:calcu}
\sum_{n\ge 1}\Lambda(n)\chi(n)\euB(n/X)
=\frac{\chi(-1)\tau(\chi)}{q}\ssum{0<m<q\\(m,q)=1}\overline\chi(m)
\sum_{n\ge 1}\Lambda(n)\e(-nm/q)\euB(n/X).
\ee
Combining \eqref{eq:superbound} and Theorem~\ref{thm:vonMangoldt-twist}, the estimate
\be\label{eq:lator}
\sum_{n\ge 1}\Lambda(n)\e(-nm/q)\euB\Big(\frac{n}{X}\Big)
=-F_{q,\euB}(X)+O_{q,\euB,\eps}(X^{9/10+\eps})
\ee
holds uniformly for $0<m<q$ with $(m,q)=1$, where 
$$
F_{q,\euB}(X)\defeq\frac{\mu(q)}{\phi(q)}\sum_{n\ge 1}\Lambda(n)\euB(n/X).
$$
We insert \eqref{eq:lator} into \eqref{eq:calcu}. Since $F_{q,\euB}(X)$ is
independent of $m$, and
$$
\ssum{0<m<q\\(m,q)=1}\overline\chi(m)=0
$$
for the nonprincipal character $\chi$, we derive the bound
$$
\sum_{n\ge 1}\Lambda(n)\chi(n)\euB(n/X)
\,\mathop{\ll}\limits_{q,\euB,\eps}\, X^{9/10+\eps}.
$$
Since $\euB\in C_c^\infty(\R^+)$ is arbitrary, Lemma~\ref{lem:ultraclean}
shows that ${\rm GRH}[\frac{9}{10},\chi]$ is true, and we are done.
\end{proof}

\begin{proof}[Proof of Theorem~\ref{thm:RHvsGRH2}]
Assume RH and ${\rm GRH}[\frac{9}{10}]$. For any
$\euB\in C_c^\infty(\R^+)$ and $\xi\defeq m/q$ with $0<m<q$ and $(m,q)=1$,
we have by Theorem~\ref{thm:vonMangoldt-twist}:
$$
\sum_{n\ge 1}\Lambda(n)\e(-n\xi)\euB\Big(\frac{n}{X}\Big)
=-\ssum{\rho=\frac12+i\gamma}\xi^{-1/2-i\gamma}
\euZ(\rho)\euB\Big(\frac{\gamma}{2\pi\xi X}\Big)
+O_{\xi,\euB}(X^{9/10}),
$$
and by Lemma~\ref{lem:aloevera}:
$$
\sum_{n\ge 1}\Lambda(n)\e(-n\xi)\euB(n/X)
=\frac{\mu(q)}{\phi(q)}\sum_{n\ge 1}\Lambda(n)\euB(n/X)
+O_{\xi,\euB,\eps}(X^{9/10+\eps}).
$$
Combining these results, we obtain \eqref{eq:superbound}.
\end{proof}

{\large\section{Acknowledgements}}

The author thanks Henryk Iwaniec for some stimulating conversations.
Many thanks also to the referees for their helpful comments and for
pointing out flaws in the original version of the manuscript.

\end{document}